\documentclass[11pt,a4paper,oneside,english,british]{amsart}

\usepackage[T1]{fontenc}
\usepackage[applemac]{inputenc}
\usepackage{babel}
\usepackage{textcomp}
\usepackage{amstext}
\usepackage{amsthm}
\usepackage{multicol}
\usepackage{amssymb}
\usepackage{stmaryrd}
\usepackage[unicode=true,
 bookmarks=false,
 breaklinks=false,pdfborder={0 0 1},backref=section,colorlinks=false]
 {hyperref}

\makeatletter

\pdfpageheight\paperheight
\pdfpagewidth\paperwidth

\numberwithin{equation}{section}
\numberwithin{figure}{section}
\theoremstyle{plain}
\newtheorem{thm}{\protect\theoremname}
\theoremstyle{definition}
\newtheorem{defn}[thm]{\protect\definitionname}
\theoremstyle{plain}
\newtheorem{prop}[thm]{\protect\propositionname}
\theoremstyle{remark}
\newtheorem{rem}[thm]{\protect\remarkname}
\theoremstyle{plain}
\newtheorem{lem}[thm]{\protect\lemmaname}
\theoremstyle{conjecture}

\theoremstyle{plain}

\usepackage{pgf}
\usepackage{tikz}\usetikzlibrary{arrows}

\usepackage{mathrsfs}

\usepackage{babel}
\addto\captionsbritish{\renewcommand{\propositionname}{Proposition}}
\addto\captionsbritish{\renewcommand{\theoremname}{Theorem}}
\addto\captionsenglish{\renewcommand{\propositionname}{Proposition}}
\addto\captionsenglish{\renewcommand{\theoremname}{Theorem}}
\providecommand{\propositionname}{Proposition}
\providecommand{\theoremname}{Theorem}

\AtBeginDocument{
  
}

\makeatother

\addto\captionsbritish{\renewcommand{\definitionname}{Definition}}
\addto\captionsbritish{\renewcommand{\lemmaname}{Lemma}}
\addto\captionsbritish{\renewcommand{\propositionname}{Proposition}}
\addto\captionsbritish{\renewcommand{\remarkname}{Remark}}
\addto\captionsbritish{\renewcommand{\theoremname}{Theorem}}
\addto\captionsenglish{\renewcommand{\definitionname}{Definition}}
\addto\captionsenglish{\renewcommand{\lemmaname}{Lemma}}
\addto\captionsenglish{\renewcommand{\propositionname}{Proposition}}
\addto\captionsenglish{\renewcommand{\remarkname}{Remark}}
\addto\captionsenglish{\renewcommand{\theoremname}{Theorem}}
\addto\captionsenglish{\renewcommand{\theoremname}{Conjecture}}
\providecommand{\definitionname}{Definition}
\providecommand{\lemmaname}{Lemma}
\providecommand{\propositionname}{Proposition}
\providecommand{\remarkname}{Remark}
\providecommand{\theoremname}{Theorem}
\providecommand{\theoremname}{Conjecture}

\begin{document}
\title[Dynamical system related to epidemic SISI model]{A non-linear discrete-time dynamical system related to epidemic SISI model}
\author{s. k. shoyimardonov }
\begin{abstract}
We consider SISI epidemic model with discrete-time. The crucial point of this model is that an
individual can be infected twice. This non-linear evolution operator depends on seven parameters and
we assume that the population size under consideration is constant, so death rate is the
same with birth rate per unit time. Reducing to quadratic stochastic operator (QSO)  we study the dynamical system of the SISI model.

\end{abstract}


\address{Sobirjon Shoyimardonov. \ \ V.I.Romanovskiy institute of mathematics, 81, Mirzo Ulug'bek str.,
100125, Tashkent, Uzbekistan.}
\email{shoyimardonov@inbox.ru}

\maketitle
\selectlanguage{british}%

\section{Introduction}

In \cite{Green}  SISI model is considered  in continuous time as a spread of
bovine respiratory syncytial virus (BRSV) amongst cattle.
They performed an equilibrium and stability analysis and considered an applications to Aujesky's disease (pseudorabies virus) in pigs.
In \cite{Muller} SISI model was considered as an example and characterised the conditions for fixed point equation.
In the both these works it was assumed that the population size under consideration is a constant,
so the per capita death rate is equal to per capita birth rate.

 Let us consider SISI model \cite{Muller}:
\begin{equation}
\begin{cases}
\frac{dS}{dt} & =b(S+I+S_1+I_1)-\mu S-\beta_1 A(I,I_1)S\\[2mm]
\frac{dI}{dt} & =-\mu I+\beta_1 A(I,I_1)S-\alpha I\\[2mm]
\frac{dS_1}{dt} & =-\mu S_1+\alpha I- \beta_2 A(I,I_1)S_1\\[2mm]
\frac{dI_1}{dt} & =-\mu I_1+\beta_2 A(I,I_1)S_1
\end{cases}\label{eq:Eq1}
\end{equation}
where
$S-$  density of susceptibles who did not have the disease before,
$I-$  density of first time infected persons,
$S_1-$  density of recovereds,
$I_1-$  density of second time infected persons,
$b-$ birth rate,
$\mu-$  death rate,
$\alpha-$ recovery rate,
$\beta_1-$  susceptibility of persons in $S$,
$\beta_2-$  susceptibility of persons in $S_1,$
$k_1-$  infectivity of persons in $I$,
$k_2-$  infectivity of persons in $I_1.$
Moreover, $A(I,I_1)$ denotes the so-called force of infection,
$$A(I,I_1)=\frac{k_1I+k_2I_1}{P}$$
and $P=S+I+S_1+I_1$ denotes the total population size. Here we do some replacements:
$$x=\frac{S}{P}, u=\frac{I}{P}, y=\frac{S_1}{P}, v=\frac{I_1}{P}$$
In (\ref{eq:Eq1}) we assume that $b=\mu$ and by substituting $x,u,y,v$ we have
\begin{equation}
\begin{cases}
\frac{dx}{dt} & =b-bx-\beta_1 A(u,v)x\\[2mm]
\frac{du}{dt} & =-bu+\beta_1 A(u,v)x-\alpha u\\[2mm]
\frac{dy}{dt} & =-by+\alpha u -\beta_2 A(u,v)y\\[2mm]
\frac{dv}{dt} & =-bv+\beta_2 A(u,v)y
\end{cases}\label{cont}
\end{equation}
where all parameters are non-negative. We notice that $\frac{d}{dt}\left(x+u+y+v\right)=0$, from
this we deduce that the total population size is constant over time and therefore we assume $x+u+y+v=1$.

\section{Quadratic Stochastic Operators}

\emph{The quadratic stochastic operator} (QSO) \cite{GMR}, \cite{L} is
a mapping of the standard simplex.
\begin{equation}
S^{m-1}=\{x=(x_{1},...,x_{m})\in\mathbb{R}^{m}:x_{i}\geq0,\sum\limits _{i=1}^{m}x_{i}=1\}\label{2}
\end{equation}
into itself, of the form
\begin{equation}
V:x'_{k}=\sum\limits _{i=1}^{m}\sum\limits _{j=1}^{m}P_{ij,k}x_{i}x_{j},\qquad k=1,...,m,\label{3}
\end{equation}
where the coefficients $P_{ij,k}$ satisfy the following conditions
\begin{equation}
P_{ij,k}\geq0,\quad P_{ij,k}=P_{ji,k},\quad\sum\limits _{k=1}^{m}P_{ij,k}=1,\qquad(i,j,k=1,...,m).\label{4}
\end{equation}

Thus, each quadratic stochastic operator $V$ can be uniquely
defined by a cubic matrix $\mathbb{P}=(P_{ij,k})_{i,j,k=1}^{m}$
with conditions (\ref{4}).

Note that each element $x\in S^{m-1}$ is a probability distribution
on $\left\llbracket 1,m\right\rrbracket =\{1,...,m\}.$
Each such distribution can be interpreted as a state of the corresponding
biological system.

For a given $\lambda^{(0)}\in S^{m-1}$ the \emph{trajectory}
(orbit) $\{\lambda^{(n)};n\geq0\}$ of $\lambda^{(0)}$ under
the action of QSO (\ref{3}) is defined by
\[
\lambda^{(n+1)}=V(\lambda^{(n)}),\;n=0,1,2,...
\]

The main problem in mathematical biology consists in the study of the asymptotical behaviour of the trajectories. The
difficulty of the problem depends on given matrix $\mathbb{P}$.

\begin{defn}
A QSO $V$ is called regular if for any initial point $\lambda^{(0)}\in S^{m-1}$,
the limit
\[
\lim_{n\to\infty}V^{n}(\lambda^{(0)})
\]
exists, where $V^n$ denotes $n$-fold composition of $V$ with itself (i.e. $n$ time iterations of $V$).
\end{defn}

\section{Reduction to QSO}

In this paper we study the discrete time dynamical system associated to the system
(\ref{cont}).

Define the evolution operator
 $
V:S^{3}\rightarrow\mathbb{R}^{4},\quad\left(x,u,y,v\right)\mapsto\left(x^{(1)},u^{(1)},y^{(1)},v^{(1)}\right)
$

\begin{equation}
V:\left\{ \begin{alignedat}{1}x^{(1)} & =x+b-bx-\beta_1 A(u,v)x\\
u^{(1)} & =u-bu+\beta_1 A(u,v)x-\alpha u\\
y^{(1)} & =y-by+\alpha u -\beta_2 A(u,v)y\\
v^{(1)} & =v-bv+\beta_2 A(u,v)y
\end{alignedat}
\right.\label{disc}
\end{equation}
where $A(u,v)=k_1u+k_2v.$
Note that if $k_1=k_2=0$ then $A(u,v)=0$ and operator (\ref{disc}) becomes linear operator which is well studied. 

By definition the operator $V$ has a form of QSO, but the
parameters of this operator are not related to $P_{ij,k}$. Here
to make some relations with $P_{ij,k}$ we find conditions on
parameters of (\ref{disc}) rewriting it in the form (\ref{3})
(as in \cite{RSH},\cite{RSHV}). Using $x+u+y+v=1$ we change
the form of the operator (\ref{disc}) as following:
\[
V:\left\{ \begin{alignedat}{1}x^{(1)} & =x(1-b)(x+u+y+v)+b(x+u+y+v)^2-\beta_1(k_{1}u+k_{2}v)x\\
u^{(1)} & =u(1-b-\alpha)(x+u+y+v)+\beta_1(k_{1}u+k_{2}v)x\\
y^{(1)} & =y(1-b)(x+u+y+v)+\alpha u(x+u+y+v) -\beta_2 (k_{1}u+k_{2}v)y\\
v^{(1)} & =v(1-b)(x+u+y+v)+\beta_2 (k_{1}u+k_{2}v)y
\end{alignedat}
\right.
\]
From this system and QSO (\ref{3}) for the case $m=4$ we obtain the
following relations:

\begin{equation}
\begin{array}{cccc}
\begin{aligned}{\scriptstyle P_{11,1}} & ={\scriptstyle 1},\\
{\scriptstyle 2P_{14,1}} & ={\scriptstyle 1+b-\beta_{1}k_{2}},\\
{\scriptstyle 2P_{24,1}} & ={\scriptstyle 2b},\\
{\scriptstyle P_{44,1}} & ={\scriptstyle b},\\
{\scriptstyle P_{22,2}} & ={\scriptstyle 1-b-\alpha},\\
{\scriptstyle 2P_{12,3}} & ={\scriptstyle \alpha},\\
{\scriptstyle 2P_{23,3}} & ={\scriptstyle 1-b+\alpha-\beta_{2}k_{1}},\\
{\scriptstyle 2P_{34,3}} & ={\scriptstyle 1-b-\beta_{2}k_{2}},\\
{\scriptstyle 2P_{24,4}} & ={\scriptstyle 1-b},\\

\end{aligned}

\begin{aligned}{\scriptstyle 2P_{12,1}} & ={\scriptstyle 1+b-\beta_{1}k_{1}},\\
 {\scriptstyle P_{22,1}} & ={\scriptstyle b},\\
 {\scriptstyle P_{33,1}} & ={\scriptstyle b},\\
 {\scriptstyle 2P_{12,2}} & ={\scriptstyle 1-b-\alpha+\beta_{1}k_{1}},\\
{\scriptstyle 2P_{23,2}} & ={\scriptstyle 1-b-\alpha},\\
{\scriptstyle 2P_{13,3}} & ={\scriptstyle 1-b},\\
{\scriptstyle 2P_{24,3}} & ={\scriptstyle \alpha},\\
{\scriptstyle 2P_{14,4}} & ={\scriptstyle 1-b},\\
{\scriptstyle 2P_{34,4}} & ={\scriptstyle 1-b+\beta_{2}k_{2}},\\

\end{aligned}

 & \begin{aligned}{\scriptstyle 2P_{13,1}} & ={\scriptstyle 1+b},\\
 {\scriptstyle 2P_{23,1}} & ={\scriptstyle 2b},\\
 {\scriptstyle 2P_{34,1}} & ={\scriptstyle 2b},\\
{\scriptstyle 2P_{14,2}} & ={\scriptstyle \beta_{1}k_{2}},\\
{\scriptstyle 2P_{24,2}} & ={\scriptstyle 1-b-\alpha},\\
{\scriptstyle P_{22,3}} & ={\scriptstyle \alpha},\\
{\scriptstyle P_{33,3}} & ={\scriptstyle 1-b},\\
{\scriptstyle 2P_{23,4}} & ={\scriptstyle \beta_{2}k_{1}},\\
{\scriptstyle P_{44,4}} & ={\scriptstyle 1-b},\\

\end{aligned}
\end{array}\label{par}
\end{equation}

${\scriptstyle \text{other }} {\scriptstyle P_{ij,k}=0}.$

\begin{prop}\label{pc}
We have $V\left(S^{3}\right)\subset S^{3}$ if and only if the
non-negative parameters $b,\alpha, \beta_1, \beta_2, k_1, k_2$
verify the following conditions
\begin{equation}
\begin{array}{cccc}
\alpha+b\leq1, & \beta_{1}k_{2}\leq2, & \beta_{2}k_{1}\leq2,\medskip\\
b+ \beta_{2}k_{2}\leq1, & \left|b-\beta_{1}k_{1}\right|\leq1, & \left|b-\beta_{2}k_{2}\right|\leq1,\medskip\\
\left|b-\beta_{1}k_{2}\right|\leq1,  & \left|\alpha+b-\beta_{1}k_{1}\right|\leq1, & \left|\alpha-b-\beta_{2}k_{1}\right|\leq1.
\end{array}\label{cond}
\end{equation}

Moreover, under conditions (\ref{cond}) the operator $V$ is
a QSO.
\end{prop}

\begin{proof} The proof can be obtained  by using equalities (\ref{par}) and solving
inequalities $0\leq P_{ij,k}\leq1$ for each $P_{ij,k}$.
\end{proof}

\begin{rem}
In the sequel of the paper we consider operator (\ref{disc})
with parameters $b,\alpha, \beta_1, \beta_2, k_1, k_2$ which satisfy conditions
(\ref{cond}). This operator maps $S^{3}$ to itself and we are
interested to study the behaviour of the trajectory of any initial
point $\lambda\in S^{3}$ under iterations of the operator $V.$
\end{rem}

\section{Fixed points of the operator (\ref{disc})}

To find fixed points of operator $V$ given by (\ref{disc})
we have to solve $V(\lambda)=\lambda.$

\subsection{Finding fixed points of the operator (\ref{disc})}
Denote
\begin{itemize}
\item[] $\lambda_{1}=\left(1,0,0,0\right), \ \  \lambda_{2}=\left(0,0,0,1\right), \ \ \lambda_{3}=\left(0,0,1,0\right),\ \ \lambda_{4}=\left(0,1,0,0\right),$
\item[] $\Lambda_{5}=\{\lambda=(x,u,y,v)\in S^3: u=v=0\},$
\item[] $\Lambda_{6}=\{\lambda=(x,u,y,v)\in S^3: u=0\},$
\item[] $\Lambda_{7}=\{\lambda=(x,u,y,v)\in S^3: x=0\},$
\item[] $\Lambda_{8}=\{\lambda=(x,u,y,v)\in S^3: x=u=0\},$
\item[]  $\lambda_{9}=\left(\frac{b}{\beta_1 k_1},\frac{\beta_1 k_1-b}{\beta_1 k_1},0,0\right), \ \ \lambda_{10}=\left(\frac{b+\alpha}{\beta_1 k_1},\frac{b(\beta_1k_1-b-\alpha)}{\beta_1k_1(b+\alpha)},\frac{\alpha(\beta_1k_1-b-\alpha)}{\beta_1k_1(b+\alpha)},0 \right),$

\item[] $\lambda_{11}=\left(\frac{b}{b+\beta_1 A},\frac{b\beta_1A}{(b+\beta_1 A)(b+\alpha)},\frac{\alpha b\beta_1A}{(b+\beta_1 A)(b+\beta_2A)(b+\alpha)},\frac{\alpha \beta_1\beta_2 A^2}{(b+\beta_1 A)(b+\beta_2A)(b+\alpha)} \right),$

where  $A$ is a positive solution of the equation
\begin{equation}\label{fpc}
1=\frac{b\beta_1k_1}{(b+\beta_1 A)(b+\alpha)}+\frac{\alpha \beta_1\beta_2k_2A}{(b+\beta_1 A)(b+\beta_2A)(b+\alpha)}
\end{equation}
\end{itemize}

By the following proposition we give all possible fixed points
of the operator $V.$
\begin{prop}\label{fixp}
\label{fp} Let $Fix(V)$ be set of fixed points of the operator (\ref{disc}). Then

$$Fix(V)=\left\{\begin{array}{lll}
\{\lambda_1\} \ \ \\[2mm]
\{\lambda_1, \lambda_2, \lambda_3\}, \ \ {\rm if} \ \  b=0 \\[2mm]
\{\lambda_2, \lambda_4\}\bigcup\Lambda_{5}, \ \ \ \ {\rm if} \ \ b=\alpha=0 \\[2mm]
\Lambda_{6}, \ \ \ \ \ \ \ \ \ \ \  {\rm if} \ \ b=\beta_1=\beta_2=0 \\[2mm]
\{\lambda_1\}\bigcup\Lambda_{7}, \ \ \ \ {\rm if} \ \ b=\alpha=\beta_2=0, \beta_1>0 \\[2mm]
\{\lambda_1,\lambda_4\}\bigcup\Lambda_8, \ \ \ \ {\rm if} \ \ b=\beta_2=0, \beta_1>0, \alpha>0, k_1k_2>0 \\[2mm]
S^3,\ \ \ \ {\rm if} \ \ b=\alpha=k_1=k_2=0  \ \ {\rm or} \ \ b=\alpha=\beta_1=\beta_2=0\\[2mm]
\{\lambda_1, \lambda_{9} \}, \ \ \ \ {\rm if} \ \ b>0, \alpha=0, \beta_1k_1>b \\[2mm]
\{\lambda_1, \lambda_{10} \}, \ \ \ \ {\rm if} \ \ b>0,\, \alpha>0,\, \beta_2=0, \, \beta_1 k_1>b+\alpha\\[2mm]
\{\lambda_1, \lambda_{11} \}, \ \ \ \ {\rm if} \ \ \alpha b \beta_1\beta_2 k_1k_2>0
\end{array}\right.$$
\end{prop}

\begin{proof}
Recall that a fixed point of the operator $V$ is a solution
of $V(\lambda)=\lambda.$ From this straightforward $\lambda_i, i=\overline{1,8}.$

For the other cases we assume that $b>0.$\\
 Now we find $\lambda_{9}.$ If $y=v=0, \alpha=0$, then by (\ref{disc}) we have $y^{(1)}=0, v^{(1)}=0,$  and $A(u,v)=k_1 u.$  Using this and $u^{(1)}=u$ we obtain $u=\frac{\beta_1 k_1-b}{\beta_1 k_1}.$ By substituting  them to $x^{(1)}=x$ we get $x=\frac{b}{\beta_1 k_1}.$ Of course, for positiveness of $u$ it requests that $\beta_1 k_1>b>0.$ Similarly, by the conditions to parameters we can find easily the next fixed point $\lambda_{10}.$

For the interior fixed point $\lambda_{11}$ we request that all parameters are positive. First, using $x^{(1)}=x$ we have $x=\frac{b}{b+\alpha},$ from this and by $u^{(1)}=u$ we get $u=\frac{\beta_1Ax}{b+\alpha}=\frac{b\beta_1A}{(b+\beta_1A)(b+\alpha)}.$ Similarly, by $y^{(1)}=y$ we have $y=\frac{\alpha u}{b+\beta_2A}=\frac{\alpha b\beta_1A}{(b+\beta_1A)(b+\beta_2A)(b+\alpha)},$ and from $v^{(1)}=v$ we get $v=\frac{\beta_2Ay}{b}=\frac{\alpha \beta_1\beta_2A^2}{(b+\beta_1A)(b+\beta_2A)(b+\alpha)}.$ In this case from $A(u,v)=k_1u+k_2v$ we obtain the quadratical equation (\ref{fpc}). Thus, Proposition is proved.
\end{proof}

Note that the set of positive solutions of (\ref{fpc}) is non-empty when $\beta_1k_1\geq b+\alpha$ (see the statements after Conjecture 2). For example,  $\alpha=0.3, b=0.2, \beta_1=0.6, \beta_2=0.4, k_1=k_2=1.$ Then the equation (\ref{fpc}) has the form
$$30A^2-5A-1=0$$ 
and the positive solution is $A=\frac{5+\sqrt{145}}{60}\approx0.284.$

\subsection{Type of the fixed point $\lambda_1$}
\begin{defn}
\cite{De}. A fixed point $p$ for $F:\mathbb{R}^{m}\rightarrow\mathbb{R}^{m}$
is called \emph{hyperbolic} if the Jacobian matrix $\textbf{J}=\textbf{J}_{F}$
of the map $F$ at the point $p$ has no eigenvalues on the unit
circle.

There are three types of hyperbolic fixed points:

(1) $p$ is an attracting fixed point if all of the eigenvalues
of $\textbf{J}(p)$ are less than one in absolute value.

(2) $p$ is an repelling fixed point if all of the eigenvalues
of $\textbf{J}(p)$ are greater than one in absolute value.

(3) $p$ is a saddle point otherwise.
\end{defn}

\begin{prop}\label{fp1} Let $\lambda_1$ be the fixed point of the operator $V.$ Then
$$\lambda_{1}=\left\{\begin{array}{lll}
{\rm nonhyperbolic}, \ \ {\rm if} \ \  b=0 \ \ {\rm or} \ \ \beta_1k_1=b+\alpha\\[2mm]
{\rm attractive}, \ \ \ \ {\rm if} \ \ b>0 \ \ {\rm and} \ \ \beta_1k_1<b+\alpha\\[2mm]
{\rm saddle}, \ \ \  \ \  {\rm if}  \ \ b>0 \ \ {\rm and} \ \ \beta_1k_1>b+\alpha
\end{array}\right.$$

\end{prop}

\begin{proof}
The Jacobian
of the operator (\ref{disc}) is:

\[ J=
\left[\begin{array}{cccccc}
1-b-\beta_1A & -\beta_1k_1x &0 & -\beta_1k_2x\\
\beta_1A & 1-b-\alpha+\beta_1k_1x & 0 & \beta_1k_2x \\
0 & \alpha-\beta_2k_1y & 1-b-\beta_2A& -\beta_2k_2y \\
0 & \beta_2k_1y & \beta_2A& 1-b+\beta_2k_2y
\end{array}\right]
\]

Then at the fixed point $\lambda_1$ the Jacobian is
\[ J(\lambda_1)=
\left[\begin{array}{cccccc}
1-b & -\beta_1k_1 &0 & -\beta_1k_2\\
0 & 1-b-\alpha+\beta_1k_1 & 0 & \beta_1k_2 \\
0 & \alpha & 1-b& 0 \\
0 & 0 & 0& 1-b
\end{array}\right]
\]
and the eigenvalues of this matrix are $\mu_1=1-b, \mu_2=1-b-\alpha+\beta_1k_1.$ By the conditions (\ref{cond}) we have $\mu_1\geq0, \mu_2\geq0.$ It is easy to se that if $b=0$ or $\beta_1k_1=b+\alpha$ then $\mu_1=1$ or $\mu_2=1$ respectively, if $b>0, \beta_1k_1<b+\alpha$ then the fixed point $\lambda_1$ is an attracting, otherwise saddle point.
\end{proof}
\begin{rem} The type of other fixed points is not studied and the type of fixed point $\lambda_1$ will be useful for some results in below.
\end{rem}
\section{The limit points of trajectories}

In this section we study the limit behavior of trajectories of initial point $\lambda^{(0)}\in S^3$ under
 operator (\ref{disc}), i.e the sequence $V^n(\lambda^{(0)})$, $n\geq 1$. Note that since $V$ is a continuous operator,   its trajectories have as a limit some fixed points obtained in Proposition \ref{fp}.

\subsection{Case no susceptibility of persons ($\beta_1=\beta_2=0$).}

We study here the case where in the model there is no susceptibility of persons.
\begin{prop}\label{prop1}
For an initial point $\lambda^{0}=\left(x^{0},u^{0},y^{0},v^{0}\right)\in S^{3}$
(except fixed points)the trajectory (under action of operator (\ref{disc})) has the following limit
\[
\lim_{n\to\infty}V^{(n)}(\lambda^{0})=\begin{cases}
\lambda^{0}  & \text{if } \alpha=b=0\\
(x^{0},0,1-x^0-v^{0}, v^0) & \text{if } b=0, \alpha>0\\
\lambda_{1} & \text{if } b>0\\
\end{cases}
\]
\end{prop}

\begin{proof}
If $\beta_1=\beta_2=0$ then the operator (\ref{disc}) has the following form:

\begin{equation}
V:\left\{ \begin{alignedat}{1}x^{(1)} & =x+b-bx\\
u^{(1)} & =u(1-b-\alpha)\\
y^{(1)} & =y-by+\alpha u \\
v^{(1)} & =v(1-b)
\end{alignedat}
\right.\label{case1}
\end{equation}
If $b=\alpha=0$ then every point is fixed point, so this case is clear. If $b=0, \alpha>0$ then by (\ref{case1}) we get $x^{(n)}=x^0, v^{(n)}=v^0$ and $u^{(n)}=u^0(1-\alpha)^n\rightarrow0.$ Moreover, $y^{(1)}=y+\alpha u\geq y,$ so the sequence $y^{(n)}$ has a limit.
From $x^{(n)}+u^{(n)}+y^{(n)}+v^{(n)}=1$ it follows the proof of this case. If $b>0$ then the sequences $u^{(n)}, v^{(n)}$  have zero limits. In addition, from $x^{(n+1)} =x^{(n)}+b(1-x^{(n)})$ one obtains that limit of the sequence $x^{(n)}$ is 1 (since $b>0$). Thus, the Proposition is proved.
\end{proof}
\subsection{Case no susceptibility of persons in $S$ ($\beta_1=0, \beta_2>0$).}
\begin{prop}
For an initial point $\lambda^{0}=\left(x^{0},u^{0},y^{0},v^{0}\right)\in S^{3}$
(except fixed points) the trajectory (under action of the operator (\ref{disc})) has the following limit
\[
\lim_{n\to\infty}V^{(n)}(\lambda^{0})=\begin{cases}
(x^{0},u^{0},0, 1-x^0-u^0)  & \text{if } \alpha=b=0\\
\lambda_{1} & \text{if } b>0, \alpha=0\\
(x^{0},0,\bar{y}, 1-x^0-\bar{y}) & \text{if } b=0, \alpha>0, k_2=0\\
(x^{0},0,0, 1-x^0) & \text{if } b=0, \alpha>0, k_2>0\\
\lambda_{1} & \text{if } b>0, \alpha>0\\
\end{cases}
\]
where $\bar{y}=\bar{y}(\lambda^0)$
\end{prop}

\begin{proof}
If $\beta_1=0$ then the operator (\ref{disc}) is

\begin{equation}
V:\left\{ \begin{alignedat}{1}x^{(1)} & =x+b(1-x)\\
u^{(1)} & =u(1-b-\alpha)\\
y^{(1)} & =y-by+\alpha u-\beta_2(k_1u+k_2v)y \\
v^{(1)} & =v-bv+\beta_2(k_1u+k_2v)y
\end{alignedat}
\right.\label{case2}
\end{equation}
\textbf{Case}: $b=\alpha=0.$  We assume that $A(u^0,v^0)\neq0,$ otherwise, $A(u^{(n)},v^{(n)})=0, \forall n\in N,$ and limit point of the operator (\ref{case2}) is initial point $\lambda^0.$ The proof of this case is coincides with second case of Proposition \ref{prop3}.\\
\textbf{Case:} $b>0, \alpha=0.$ In this case $u^{(n)}=u^0(1-b)^n$ has zero limit, and we have
$$x^{(1)}=x+b(1-x)\geq x, \ \ y^{(1)} =y-by-\beta_2(k_1u+k_2v)y\leq y,$$ so the sequences $x^{(n)},y^{(n)}$ have limits and consequently, $v^{(n)}$ also has limit. Let $\bar{x}$ be a limit of $x^{(n)}.$ Then from $x^{(n+1)}=x^{(n)}+b(1-x^{(n)})$ we get limit and it follows that $\bar{x}=1$ since $b>0.$ Thus, limit of the considering operator is $\lambda_1=(1,0,0,0).$\\
\textbf{Case:} $b=0, \alpha>0, k_2=0.$  Then $x^{(n)}=x^0, u^{(n)}=u^0(1-\alpha)^n\rightarrow0$  and $v^{(1)}=v+\beta_2k_1uy\geq v,$ i.e., the sequence $v^{(n)}$ has limit, so $y^{(n)}$ also has limit. But limits of  $y^{(n)}$ and  $v^{(n)}$ depend on initial point $\lambda^0.$\\
\textbf{Case:} $b=0, \alpha>0, k_2>0.$ Here also, as previous case, $x^{(n)}=x^0, u^{(n)}=u^0(1-\alpha)^n\rightarrow0$ and the sequences $y^{(n)}, v^{(n)}$ have limits. Let $\bar{y}, \bar{v}$ be limits of $y^{(n)}$ and $v^{(n)}$ respectively. If we take limit from both side of the following equality
$$v^{(n+1)}=v^{(n)}+\beta_2(k_1u^{(n)}+k_2v^{(n)})y^{(n)}$$ then we have $\beta_2\bar{v}\bar{y}=0,$ i.e., $\bar{y}=0,$ because,  $v^{(n)}$ increasing sequence, so $\bar{v}\neq 0$ (Note that $A(u^0,v^0)\neq0$). Thus, $\bar{v}=1-x^0.$\\
\textbf{Case:} $b>0, \alpha>0.$ Then $u^{(n)}=u^0(1-b-\alpha)^n\rightarrow0,$ and
$x^{(1)}=x+b(1-x)\geq x\geq x,$ i.e., the sequence $x^{(n)}$ has limit $\bar{x}$. From $x^{(n+1)}=x^{(n)}+b(1-x^{(n)})$ we get limit and it obtains that $\bar{x}=1.$ Moreover, from the $x^{(n)}+u^{(n)}+y^{(n)}+ v^{(n)}=1$ we have that  $y^{(n)}+v^{(n)}\rightarrow0.$  In addition, every terms of the both sequences are non-negative, so the limits of the sequences $y^{(n)}$ and $v^{(n)}$ exist and zero. Thus, the proof of the Proposition is completed.
\end{proof}

\subsection{Case no birth (death) rate and recovery rate ($b=\alpha=0$).}
\begin{prop}\label{prop3}
For an initial point $\lambda^{0}=\left(x^{0},u^{0},y^{0},v^{0}\right)\in S^{3}$
(except fixed points) the trajectory (under action of the operator (\ref{disc})) has the following limit
\[
\lim_{n\to\infty}V^{(n)}(\lambda^{0})=\begin{cases}
\lambda^{0}  & \text{if } k_1=k_2=0\\
(x^{0},u^{0},0, 1-x^0-u^0) & \text{if } \beta_1=0, \beta_2>0, k_1+k_2>0\\
(0,1-y^0-v^0,y^0,v^0) & \text{if }  \beta_1>0, \beta_2=0, k_1+k_2>0\\
(0,u^0,0,1-u^0) & \text{if } \beta_1>0, \beta_2>0, k_1k_2>0\\
\end{cases}
\]
\end{prop}

\begin{proof}
If $b=0, \alpha=0$ then the operator (\ref{disc}) is

\begin{equation}
V:\left\{ \begin{alignedat}{1}x^{(1)} & =x-\beta_1(k_1u+k_2v)x\\
u^{(1)} & =u+\beta_1(k_1u+k_2v)x\\
y^{(1)} & =y-\beta_2(k_1u+k_2v)y \\
v^{(1)} & =v+\beta_2(k_1u+k_2v)y
\end{alignedat}
\right.\label{case3}
\end{equation}
From the equations of the operator (\ref{case3}) we have that the sequences $x^{(n)},u^{(n)},$ $y^{(n)}, v^{(n)}$ are monotone, so they have limits. The case $k_1=k_2=0$ is clear. If $\beta_1=0, \beta_2>0, k_1+k_2>0$ then
$$x^{(n)} =x^0, u^{(n)} =u^0, y^{(n+1)} =y^{(n)}-\beta_2(k_1u^{(n)}+k_2v^{(n)})y^{(n)}.$$ If $A(u^0,v^0)=k_1u^0+k_2v^0\neq0,$ then $A(u^{(n)},v^{(n)})\neq0,$ otherwise,
$$A(u^{(n)},v^{(n)})=k_1u^{(n)}+k_2v^{(n)}=k_1u^0+k_2v^0=0,$$ so $A(u^{(n)},v^{(n)})=k_1u^{(n)}+k_2v^{(n)}$ has non-zero limit.  We assume that $\lim_{n\to\infty}y^{(n)}=\overline{y}\neq0.$ Then from $y^{(n+1)}=y^{(n)}-\beta_2(k_1u^{(n)}+k_2v^{(n)})y^{(n)}$ we get limit and it is contradiction to $\lim_{n\to\infty}A(u^{(n)},v^{(n)})\neq0,$ so we have a proof of the second case. Similarly, for cases $\beta_1>0, \beta_2=0, k_1+k_2>0$ and $\beta_1>0, \beta_2>0, k_1k_2>0$ one can complete the prove of this Proposition.
\end{proof}

\subsection{Case no recovery rate and infectivity of persons in $I_1$ ($\alpha=0, k_2=0$).}
\begin{prop}
For an initial point $\lambda^{0}=\left(x^{0},u^{0},y^{0},v^{0}\right)\in S^{3}$
(except fixed points) the trajectory (under action of the operator (\ref{disc})) has the following limit
\[
\lim_{n\to\infty}V^{(n)}(\lambda^{0})=\begin{cases}
\lambda_1  & \text{if } \beta_1k_1\leq b \ \ \text{or} \ \ u^{0}=0 \\
\lambda_{9} & \text{if }  \beta_1k_1>b \ \ \text{and} \ \ u^{0}>0\\
\end{cases}
\]
\end{prop}

\begin{proof}
Here we consider the case $b>0,$ otherwise it coincides with previous Proposition cases. If $\alpha=k_2=0$ then the operator (\ref{disc}) is

\begin{equation}
V:\left\{ \begin{alignedat}{1}x^{(1)} & =x+b-bx-\beta_1k_1ux\\
u^{(1)} & =u-bu+\beta_1k_1ux\\
y^{(1)} & =y-by-\beta_2k_1uy \\
v^{(1)} & =v-bv+\beta_2k_1uy
\end{alignedat}
\right.\label{case4}
\end{equation}
From the system (\ref{case4}) we have $y^{(1)}=y-by-\beta_2k_1uy\leq y(1-b),$ i.e., $y^{(n)} \leq y^0(1-b)^n,$ so $y^{(n)}\rightarrow0$ as $n\rightarrow\infty.$ Moreover, from the sum of last two equations of (\ref{case4}) we get
$$\lim_{n\to\infty}(y^{(n+1)}+v^{(n+1)})=(1-b)\lim_{n\to\infty}(y^{(n)}+v^{(n)})=(y^{0}+v^{0})\lim_{n\to\infty}(1-b)^n=0,$$
thus, the sequence $v^{(n)}$ converges to zero. \\

 --If $\beta_1k_1=0$ then $u^{(n)} = u^0(1-b)^n$ has zero limit and from this we have that the  sequence $x^{(n)}$ has limit one. Thus, the limit of the operator $V$ is $\lambda_1.$

 --If $0<\beta_1k_1\leq b$ then from  $u^{(1)}=u-(b-\beta_1k_1x)u\leq u$ we get that the sequence $u^{(n)}$ has limit. We assume that $\lim_{n\to\infty}u^{(n)}=\bar{u}\neq0,$ then from this and $u^{(n+1)}=u^{(n)}-(b-\beta_1k_1x^{(n)})u^{(n)}$ we have $\lim_{n\to\infty}x^{(n)}=\frac{b}{\beta_1k_1},$ and from this $\bar{u}=1-\frac{b}{\beta_1k_1}=\frac{\beta_1k_1-b}{\beta_1k_1}.$ But the condition $\beta_1k_1\leq b$ is contradiction of positiveness of $\bar{u},$ so  $u^{(n)}$ has zero limit.  Hence, limit point of the operator is $\lambda_1.$

 --If $\beta_1k_1>b$ and $u^{0}>0.$ From first two equations of the operator (\ref{case4}) we formulate a new operator:

\begin{equation}
W:\left\{ \begin{alignedat}{1}x^{(1)} & =x+b-bx-\beta_1k_1ux\\
u^{(1)} & =u-bu+\beta_1k_1ux
\end{alignedat}\label{subcase4}
\right.
\end{equation}
 Here we normalize the operator (\ref{subcase4}) as following:
 \begin{equation}
W_0:\left\{ \begin{alignedat}{1}x^{(1)} & =\frac{x+b-bx-\beta_1k_1ux}{x+u+b-b(x+u)}\\
u^{(1)} & =\frac{u-bu+\beta_1k_1ux}{x+u+b-b(x+u)}
\end{alignedat}
\right.
\end{equation}
 For this operator $x^{(n)}+u^{(n)}=1, \ \ n\geq1,$ so from $x^{(1)}+u^{(1)}=1$ we have
 \[
 \left\{\begin{alignedat}{1}x^{(2)} & =x^{(1)}+b-bx^{(1)}-\beta_1k_1u^{(1)}x^{(1)}\\
u^{(2)} & =u^{(1)}-bu^{(1)}+\beta_1k_1u^{(1)}x^{(1)}
\end{alignedat}
\right.
\]
 Thus, operators $W$ and $W_0$ have same dynamics. From first equation of the last system we obtain
  $x^{(2)}=(1-b-\beta_1k_1)x^{(1)}+\beta_1k_1(x^{(1)})^2+b.$ If we denote $x^{(1)}=x, x^{(2)}=f_{b,\beta_1k_1}(x)$ then we get
    $$f_{b,\beta_1k_1}(x)=b+(1-b-\beta_1k_1)x+\beta_1k_1x^2.$$

 \begin{defn}\label{top} (see \cite{De}, p. 47) Let $f:A\rightarrow A$ and $g:B\rightarrow B$ be two maps. $f$ and $g$ are said to be topologically conjugate if there exists a homeomorphism $h:A\rightarrow B$ such that, $h\circ f=g\circ h$. The homeomorphism $h$ is called a topological conjugacy.
 \end{defn}
 Let $F_\mu(x)=\mu x(1-x)$ be quadratic family (discussed in \cite{De}) and $f_{b,\beta_1k_1}(x)=b+(1-b-\beta_1k_1)x+\beta_1k_1x^2.$
 \begin{lem}\label{ps} Two maps $F_\mu(x)$ and $f_{b,\beta_1k_1}(x)$ are topologically conjugate for $\mu=\beta_1k_1-b+1$.
 \end{lem}
 \begin{proof} We take the linear map $h(x)=px+q$ and by Definition \ref{top} we should have $h(F_\mu(x))=f_{b,\beta_1k_1}(h(x)),$ i.e.,
 $$p\mu x(1-x)+q=b+(px+q)(1-b-\beta_1k_1)+\beta_1k_1(px+q)^2$$
 from this identity we get
 \[
\begin{cases}
-p\mu=\beta_1k_1p^2\\
p\mu=p(1-b-\beta_1k_1)+2pq\beta_1k_1\\
q=q(1-b-\beta_1k_1)+\beta_1k_1q^2+b
\end{cases}
\Rightarrow
\begin{cases}
p=-\frac{\mu}{\beta_1k_1}\\
q=\frac{\mu-1+b+\beta_1k_1}{2\beta_1k_1}\\
\beta_1k_1q^2-(b+\beta_1k_1)q+b=0
\end{cases}
\]
The roots of the equation $\beta_1k_1q^2-(b+\beta_1k_1)q+b=0$ are $q=1, q=\frac{b}{\beta_1k_1}.$ If we choose $q=1$ then by $q=\frac{\mu-1+b+\beta_1k_1}{2\beta_1k_1}$ we have $\mu=\beta_1k_1-b+1.$ Then the homeomorphism is $h(x)=\frac{b-\beta_1k_1-1}{\beta_1k_1}x+1.$ Moreover, since $b<\beta_1k_1\leq2$ we have $1<\mu<3.$
\end{proof}
The importance of this Lemma is that if two maps are topologically conjugate
then they have essentially the same dynamics (see \cite{De}, p. 53).
The operator $f_{b,\beta_1k_1}(x)=b+(1-b-\beta_1k_1)x+\beta_1k_1x^2$ has two fixed points $p_1=1$ and $p_2=\frac{b}{\beta_1k_1}.$ In addition, $f'_{b,\beta_1k_1}(x)=1-b-\beta_1k_1+2\beta_1k_1x,$ from this and $\beta_1k_1>b$ it obtains that the fixed point $p_1=1$ is repelling, $p_2=\frac{b}{\beta_1k_1}$ is attractive. Moreover, for any initial point $x\in(0,1)$ and for $\mu\in(1,3)$ the trajectory of the operator $F_{\mu}(x)$ converges to the attractive fixed point (see \cite{De}, p. 32). Thus, for the case $\beta_1k_1>b$ the limit point of the operator (\ref{case4}) is $\lambda_{9}.$
\end{proof}

\subsection{Case no only susceptibility of persons in $S_1$ ($\beta_2=0, \beta_1>0$).}
\begin{prop}\label{pro12} For an initial point $\lambda^{0}=\left(x^{0},u^{0},y^{0},v^{0}\right)\in S^{3}$
(except fixed points) the trajectory (under action of the operator (\ref{disc})) has the following limit
\[
\lim_{n\to\infty}V^{(n)}(\lambda^{0})=\begin{cases}
(x^0,0,1-x^0-v^0,v^0) & \text{if }  b=0, \alpha>0 \ \ \text{and}\ \  k_1u^0+k_2v^0=0 \\
(0,0,1-v^0,v^0) & \text{if }  b=0, \alpha>0, k_2v^0>0\\
(\bar{x},0,1-\bar{x}-v^0,v^0) & \text{if }  b=k_2v^0=0, k_1u^0>0, \alpha>0\\
\lambda_1  & \text{if } \ \ b\alpha>0, k_1u^0+k_2v^0=0 , \\
\lambda_1  & \text{if } \ \ b\alpha>0, \, k_2v^0=0, \, \beta_1k_1\leq b+\alpha, \\
\end{cases}
\]
where $\bar{x}=\bar{x}(\lambda^0)$
\end{prop}
\begin{proof}
Here we consider the case $\beta_1>0,$ otherwise, this proposition is same with Proposition \ref{prop1}. We note that the case $b=\alpha=0$ is considered in Proposition \ref{prop3}.   If $\beta_2=0$ then the operator (\ref{disc}) is

\begin{equation}
V:\left\{ \begin{alignedat}{1}x^{(1)} & =x+b-bx-\beta_1 (k_1u+k_2v)x\\
u^{(1)} & =u-bu-\alpha u+\beta_1 (k_1u+k_2v)x\\
y^{(1)} & =y-by+\alpha u\\
v^{(1)} & =v(1-b)
\end{alignedat}
\right.\label{case5}
\end{equation}
 \textbf{Case:} $b=0, \alpha>0, k_1u^0+k_2v^0=0.$ From $A(u^0,v^0)=k_1u^0+k_2v^0=0$ we have the following simple cases:

--If $k_1=k_2=0$ then $x^{(n)}=x^0, v^{(n)}=v^0$ and $u^{(n)}=u^0(1-\alpha)^n\rightarrow0.$

--If $k_1=v^0=0$ then $v^{(n)}=0,$ $x^{(n)}=x^0,$  so $u^{(n)}=u^0(1-\alpha)^n\rightarrow0$ and $y^{(n)}\rightarrow 1-x^0.$

--If $u^0=k_2=0,$ then $u^{(n)}=0, v^{(n)}=v^0,$ so $x^{(n)}=x^0, y^{(n)}=y^0=1-x^0.$

--If $u^0=v^0=0,$ then $u^{(n)}=0, v^{(n)}=0,$ so $x^{(n)}=x^0, y^{(n)}=y^0=1-x^0.$
 \textbf{Case:} $b=0, \alpha>0, k_2v^0>0$ then $v^{(n)}=v^0$ and $x^{(1)}=x-\beta_1 (k_1u+k_2v)x\leq x,$ $y^{(1)}=y+\alpha u\geq y,$ i.e., the sequences  $x^{(n)}, y^{(n)}, v^{(n)}$ have limits, so $u^{(n)}$ also has limit. From the equation $y^{(n+1)}=y^{(n)}+\alpha u^{(n)}$ we get limit and by $\alpha>0$ it obtains that $u^{(n)}$ converges to zero. Moreover, by the second equation of the system (\ref{case5}), $u^{(n+1)}=(1-\alpha)u^{(n)}+\beta_1(k_1u^{(n)}+k_2v^{(n)})x^{(n)},$ if we take a limit from two sides then we have $0=\beta_1k_2v^0\bar{x},$ from this and $k_2v^0>0$ we have $\bar{x}=0,$ where $\bar{x}$ is a limit of the sequence $x^{(n)}.$\\
  \textbf{Case:} $b=k_2v^0=0, k_1u^0>0, \alpha>0.$ Here also as previous case, all sequences have limits and $v^{(n)}=v^0$. From  $y^{(n+1)}=y^{(n)}+\alpha u^{(n)}$ we get limit and by $\alpha>0$ it obtains that $u^{(n)}$ converges to zero. But, the limit $\lim_{n\to\infty}x^{(n)}=\bar{x}$ depends on initial conditions $x^0, u^0.$\\
  \textbf{Case:} $b>0, \alpha>0, k_1u^0+k_2v^0=0.$ 
  
  --If $k_1=k_2=0$ then $x^{(1)}=x+(1-x)b\geq x, u^{(n)}=u^0(1-b-\alpha)^n\rightarrow0$ and from $v^{(n)}=v^0(1-b)^n\rightarrow0$ we get that all sequences have limits. Since $b>0$ we have that $x^{(n+1)}=x^{(n)}(1-b)+b$ has limit 1.

--If $k_1=v^0=0$ then $v^{(n)}=0,$  $u^{(n)}=u^0(1-b-\alpha)^n\rightarrow0$ and as previous case $x^{(n)}\rightarrow 1.$

--If $u^0=k_2=0$ then $u^{(n)}=0, y^{(n)}=y^0(1-b)^n\rightarrow0$ and from $v^{(n)}\rightarrow0$ implies $x^{(n)}\rightarrow 1.$

--If $u^0=v^0=0$ then $u^{(n)}=0, v^{(n)}=0,y^{(n)}=y^0(1-b)^n\rightarrow0$ so $x^{(n)}\rightarrow 1.$

\textbf{Case:} $b>0, \alpha>0, k_2v^0=0,  \beta_1k_1\leq b+\alpha.$ 

-- If  $k_2=0$ then 
$$u^{(1)}=u-(b+\alpha-\beta_1k_1x)u\leq u,$$
i.e., the sequence $u^{(n)}$ has limit. Let us to show existence the limit of $y^{(n)}.$ We assume that $y^{(1)}\leq y,$ i.e., $by-\alpha u\geq0.$ If we show that $by^{(1)}-\alpha u^{(1)}\geq0$ then it obtains that $y^{(n)}$ has limit.  We check this condition:
$$
by^{(1)}-\alpha u^{(1)}=b(y-by+\alpha u)-\alpha(u-bu-\alpha u+\beta_1k_1ux)=$$
$$=by-\alpha u+b\alpha u-b^2y+b\alpha u+\alpha^2u-\alpha \beta_1k_1ux=$$
$$=by-\alpha u-b(by-\alpha u)+\alpha u(b+\alpha-\beta_1k_1x)=$$
$$=(by-\alpha u)(1-b)+\alpha u(b+\alpha-\beta_1k_1x)\geq0.
$$
Thus, the sequences $u^{(n)}$ and $y^{(n)}$ have limits and from $v^{(n)}\rightarrow0$ we get that all sequences have limits. For the case $\beta_1k_1\leq b+\alpha$ fixed point $\lambda_1$ is unique, so it must be limit point. 

--If $v^0=0$ then the proof is same with case $k_2=0.$  Thus, the Proposition is proved. 

  For the case $b>0, \alpha>0, k_2v^0>0$ if we assume that the sequences $x^{(n)}, u^{(n)}, y^{(n)}$ have limits then for $\beta_1k_1\leq b+\alpha$ fixed point $\lambda_1$ is unique globally attracting, for  $\beta_1k_1>b+\alpha$ fixed point $\lambda_{1}$ is saddle and fixed point $\lambda_{10}$ must be limit point of the operator $V,$ because there is no other fixed points. In addition, we consider some numerical simulations for these two cases ($\beta_1k_1\leq b+\alpha$ and $\beta_1k_1>b+\alpha$). In Fig. \ref{f1} the trajectory converges to $\lambda_1,$ and in Fig \ref{f2} the trajectory converges to $\lambda_{10}.$ Therefore we formulate the following conjecture:

\textbf{Conjecture 1.}  If $\beta_2=0$ then for an initial point $\lambda^{0}=\left(x^{0},u^{0},y^{0},v^{0}\right)\in S^{3}$
(except fixed points) the trajectory has the following limit
\[
\lim_{n\to\infty}V^{(n)}(\lambda^{0})=\begin{cases}
\lambda_1  & \text{if } \ \  \ \ \beta_1k_1\leq b+\alpha, b\alpha>0 \ \ \text{and} \ \ k_2v^0>0\\
\lambda_{10} & \text{if } \ \ u^0+v^0>0 \ \ \text{and} \ \ \beta_1k_1>b+\alpha, b\alpha>0\\
\end{cases}
\]

  \begin{figure}[h!]
\begin{multicols}{2}
\hfill
\includegraphics[width=6cm]{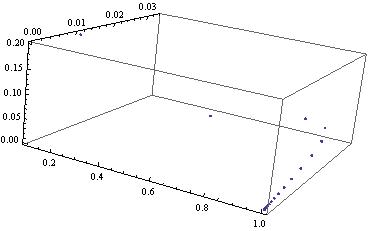}
\hfill
\caption{$\alpha=0.2, b=0.6, \beta_1=0.5, \beta_2=0, k_1=1, k_2=0.3, x^{0}=0.1, u^{0}=0.01, y^0=0.2,$ $\lim_{n\to\infty}V^{(n)}(\lambda^{0})=\lambda_1.$}\label{f1}
\label{figLeft}
\hfill
\includegraphics[width=6cm]{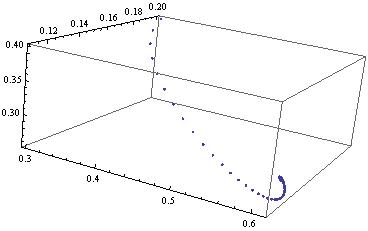}
\hfill
\caption{$\alpha=0.2, b=0.1, \beta_1=0.5, \beta_2=0, k_1=1, k_2=0.3, x^{0}=0.3, u^{0}=0.2, y^0=0.4,$ $\lim_{n\to\infty}V^{(n)}(\lambda^{0})=\lambda_{10}.$}\label{f2}
\label{figRight}
\end{multicols}
\end{figure}

  \end{proof}

The following Conjecture also formulated for the limit point of the operator (\ref{disc}) with nonzero parameters.

 \textbf{Conjecture 2.} If $\alpha b \beta_1\beta_2 k_1k_2>0$ then for an initial point $\lambda^{0}=\left(x^{0},u^{0},y^{0},v^{0}\right)\in S^{3}$
(except fixed points) the trajectory  has the following limit
\[
\lim_{n\to\infty}V^{(n)}(\lambda^{0})=\begin{cases}
\lambda_1  & \text{if } \ \ u^0=v^0=0 \ \ \text{or} \ \ \beta_1k_1\leq b+\alpha, \ \ b(b+\alpha)\geq\alpha\beta_2k_2 \\
\lambda_{11} & \text{if } \ \  u^0+v^0>0 \ \ \text{and} \ \ \beta_1k_1>b+\alpha\\
\end{cases}
\]

The case $u^0=v^0=0$ is clear, i.e.,  the limit of the operator $V$ is $\lambda_1.$

For the case  $u^0+v^0>0$ let us assume that all sequences have limits and consider limit behaviour in below.
First, we denote by $f(x), g(x):$
\begin{equation}
f(x)=b+\beta_1x, \ \ \ \ \ \  g(x)=\frac{b\beta_1k_1}{b+\alpha}+\frac{\alpha\beta_1\beta_2k_2x}{(b+\beta_2x)(b+\alpha)}.
\end{equation}
Then the roots of the equation (\ref{fpc})  are roots of the equation $f(x)=g(x).$ We consider graphical solutions of this equation.

\textbf{Case}: $\beta_1k_1>b+\alpha$. In this case  $\frac{b\beta_1k_1}{b+\alpha}>b$ and the graphic of the function $g(x)$ has horizontal asymptote $y=\frac{\beta_1(bk_1+\alpha k_2)}{b+\alpha}=const,$ so the equation $f(x)=g(x)$ has unique positive solution (Fig. \ref{f5}). Moreover, for $\beta_1k_1>b+\alpha$ fixed point $\lambda_1$ is saddle fixed point, so $\lambda_{11}$ must be limit point of the operator (\ref{disc}).

\textbf{Case}: $\beta_1k_1<b+\alpha.$ In this case, slope of the $f(x)$ is $Tan\varphi=\beta_1$ and slope of a tangent at the point $x=0$ of $g(x)$ is $Tan\psi=\frac{\alpha\beta_1\beta_2k_2}{b(b+\alpha)}.$ It is clear that, if $Tan\varphi\geq Tan\psi,$  i.e., $\beta_1\geq\frac{\alpha\beta_1\beta_2k_2}{b(b+\alpha)}$  or  $b(b+\alpha)\geq\alpha\beta_2k_2$ then  $f(x)=g(x)$ does not have positive solution (Fig.\ref{f6}).

\textbf{Case}: $\beta_1k_1=b+\alpha.$ In this case the equation $f(x)=g(x)$ has solution $x=0,$ i.e., operator (\ref{disc}) has unique fixed point $\lambda_1.$
\begin{figure}[h!]
\begin{multicols}{2}
\hfill
\includegraphics[width=6cm]{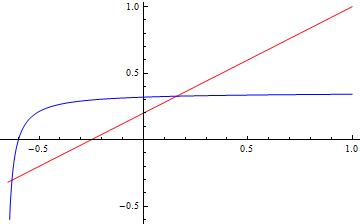}
\hfill
\caption{ $\beta_1k_1>b+\alpha$}\label{f5}
\label{figLeft}
\hfill
\includegraphics[width=6cm]{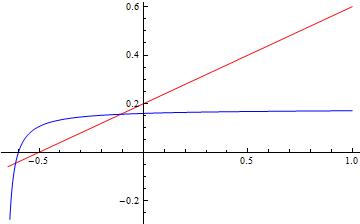}
\hfill
\caption{ $\beta_1k_1<b+\alpha, \ \ b(b+\alpha)>\alpha\beta_2k_2 $}\label{f6}
\label{figRight}
\end{multicols}
\end{figure}

\begin{figure}[h!]
\begin{multicols}{2}
\hfill
\includegraphics[width=6cm]{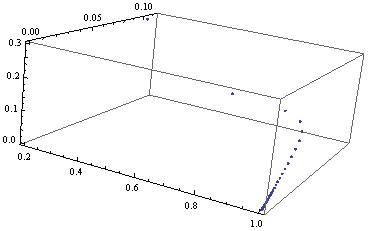}
\hfill
\caption{ $\alpha=0.1, b=0.6, \beta_1=0.5, \beta_2=0.01, k_1=1.2, k_2=1.1, x^{0}=0.2, u^{0}=0.1, y^0=0.3,$ $\lim_{n\to\infty}V^{(n)}(\lambda^{0})=\lambda_1.$}\label{f3}
\label{figLeft}
\hfill
\includegraphics[width=6cm]{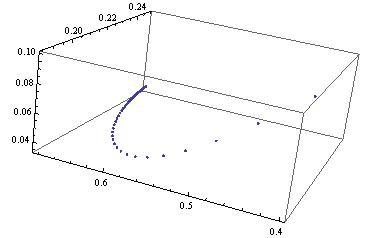}
\hfill
\caption{ $\alpha=0.01, b=0.1, \beta_1=0.8, \beta_2=0.2, k_1=0.5, k_2=1.2, x^{0}=0.2, u^{0}=0.4, y^0=0.1,$ $\lim_{n\to\infty}V^{(n)}(\lambda^{0})=\lambda_{11}.$}\label{f4}
\label{figRight}
\end{multicols}
\end{figure}

\end{document}